\documentclass[12pt]{amsart} % Document type

\usepackage[margin =1.5in]{geometry} % Set page dimensions
\usepackage{parskip} % Custom paragraph spacing
\usepackage{setspace} % Allow for double-spacing using \doublespace
\usepackage{multicol} % Allows multiple columns on a page
\usepackage{mathtools,amssymb,amsfonts,amsthm} % Variety of useful things for math
\usepackage{euscript,mathrsfs,bbm} % Some additional fonts
\usepackage{color, soul} % Color
\usepackage[usenames,dvipsnames,svgnames,table]{xcolor} % Extra colors
\usepackage{wasysym} % bunch of symbols

\usepackage{comment} % Customizable comment environments
\usepackage{marginnote} % notes in margins
\usepackage{newverbs} % Customizable verbatim environments
\usepackage{enumerate} % Customizable enumerate environments
\usepackage{tabularx} % Greater flexibility with tables
\usepackage{multirow} % Table elements span multiple rows
\usepackage{array}

\usepackage{graphicx} % For inserting images
\usepackage{wrapfig} % Wrap text around figures
\usepackage{url} % For inserting urls
\usepackage[linktoc=all,colorlinks=true,linkcolor=blue,urlcolor=Purple,citecolor=cyan]{hyperref} % Internal and external links
\usepackage[all]{hypcap} % Image captions (requires hyperref)
\usepackage{pdfpages} % Insert pdfs
\usepackage{tikz} % For drawing pictures
    \usetikzlibrary{petri}
    \usetikzlibrary{trees}

  \usepackage{pgfplots}
  \usepackage{pgfplotstable}
  \usetikzlibrary{patterns}

\allowdisplaybreaks

\specialcomment{lecturenotes}{}{}\excludecomment{lecturenotes}
\specialcomment{answers}{ \textsc{Answer: }}{}
\specialcomment{remarks}{

\setlength{\leftskip}{5mm}
\begingroup\color{blue}
\marginnote{$*****$}
}
{\endgroup

\setlength{\leftskip}{0mm}
}
\specialcomment{inremarks}{\begingroup\color{blue}}{\endgroup}

\newverbcommand{\cverb}{\small\color{purple}}{}
\newverbcommand{\everb}{\small\color{blue}}{}

\newcolumntype{L}[1]{>{\hsize=#1\hsize\raggedright\arraybackslash}X}
\newcolumntype{R}[1]{>{\hsize=#1\hsize\raggedleft\arraybackslash}X}
\newcolumntype{C}[1]{>{\hsize=#1\hsize\centering\arraybackslash}X}

\newcommand{\atext}[1]{\text{\quad #1\quad}}
\newcommand{\bbn}{\mathbb{N}}

\newcommand{\bbz}{\mathbb{Z}}

\newcommand{\lcm}{\operatorname{lcm}}

\newcommand{\ord}{\operatorname{ord}}

\newtheorem{thm}{Theorem}[section]
\newtheorem{prop}[thm]{Proposition}
\newtheorem{cor}[thm]{Corollary}
\newtheorem{question}[thm]{Question}
\newtheorem{hypothesis}[thm]{Hypothesis}
\newtheorem{lem}[thm]{Lemma}
\newtheorem*{prop*}{Proposition}
\newtheorem*{cor*}{Corollary}
\newtheorem*{lem*}{Lemma}
\newtheorem*{thm*}{Theorem}

\theoremstyle{definition}
\newtheorem{conj}[thm]{Conjecture}

\newtheorem{example}[thm]{Example}
\newtheorem*{conj*}{Conjecture}

\newcounter{homework}

\newcounter{exercise}[homework]

\theoremstyle{remark}
\newtheorem*{defn*}{Definition}

\newtheorem*{remark*}{Remark}

\author[A. S. Chen]{Annie S. Chen}
\address{Boulder High School, 1604 Arapahoe Ave, Boulder, CO 80302} 
\email{annieboulder@gmail.com}
\author[T. A. Gassert]{T. Alden Gassert}
\address{Western New England University, 1215 Wilbraham Road, Springfield, MA 01119}
\email{thomas.alden.gassert@gmail.com}
\author[K. E. Stange]{Katherine E. Stange} 
\address{Department of Mathematics, University of Colorado,
Campus Box 395, Boulder, Colorado 80309-0395}
\email{kstange@math.colorado.edu}

\title[Index divisibility in dynamical sequences and orbits mod $p$]{Index divisibility in dynamical sequences and cyclic orbits modulo $p$}
\date{\today} % Date 


\thanks{
The third author's work was supported by the National Security Agency grant H98230-16-1-0040 and National Science Foundation grant DMS-1643552. 
}

\begin{document}
\maketitle

%\excludecomment{remarks}
%\excludecomment{figure}\let\endfigure\relax

\begin{abstract}
Let $\phi(x) = x^d + c$ be an integral polynomial of degree at least 2, and consider the sequence $(\phi^n(0))_{n=0}^\infty$, which is the orbit of $0$ under iteration by $\phi$. Let $D_{d,c}$ denote the set of positive integers $n$ for which $n \mid \phi^n(0)$. We give a characterization of $D_{d,c}$ in terms of a directed graph and describe a number of its properties, including its cardinality and the primes contained therein.  In particular, we study the question of which primes $p$ have the property that the orbit of $0$ is a single $p$-cycle modulo $p$.  We show that the set of such primes is finite when $d$ is even, and conjecture that it is infinite when $d$ is odd.
\end{abstract}

\section{Introduction}

A dynamical sequence is the orbit $\alpha, \phi(\alpha), \phi^2(\alpha), \ldots$ of some $\alpha$ in a ring $R$ under iteration of a map $\phi: R \rightarrow R$.  In arithmetic dynamics, one takes $\phi$ to be a rational map defined over a number field and $\alpha$ to be an algebraic number.  Such dynamical sequences have many properties in common with their more well-known cousins: recurrence sequences and algebraic divisibility sequences arising from algebraic groups, such as Lucas sequences and elliptic divisibility sequences.  In particular, all such sequences $a_n$ are \emph{divisibility sequences}, i.e. whenever $n \mid m$, then $a_n \mid a_m$. 

The study of the primes appearing in such sequences has a centuries-long history dating back at least to Fermat's study of primes of the form $2^{2^n} +1$.  The primes appearing in a dynamical sequence encode information about the dynamical system in residue fields.  For example, taking $R = \mathbb{Z}$, if $p \mid \phi^n(0)$, then $0$ has period dividing $n$ in the dynamical system $\phi: \mathbb{Z}/p\mathbb{Z} \rightarrow \mathbb{Z}/p\mathbb{Z}$.  In particular, $p \mid \phi^p(0)$ if and only if the dynamical system given by $\phi$ on $\mathbb{Z}/p\mathbb{Z}$ consists of a single orbit of size $1$ or $p$.  Silverman studied the statistics of orbit sizes for rational maps modulo a varying prime $p$ \cite{SilvermanNewYorkJournal} (see also \cite{Chang}).  

In this paper, we restrict ourselves to the study of the maps $\phi(x) = x^d + c \in \mathbb{Z}[x]$, where $d \ge 2$. The orbit structure for $x^2+c$ is of particular interest for primality testing, integer factorization and pseudo-random number generation \cite{Blum,Lucas,Pollard}.  Silverman collected some numerical data on quadratic maps $x^2 +c$ \cite{SilvermanNewYorkJournal}, while  Peinado, Montoya, Mu\~noz and Yuste give explicit upper bounds for the cycle sizes of $x^2+c$ in a finite field \cite{Peinado}; more explicit structure is known for the exceptional maps $x^2$ and $x^2-2$ \cite{Vasiga}.   Jones \cite{j08} found that the natural density of primes dividing at least one nonzero term of a dynamical sequence is zero for four infinite families of quadratic functions, including $\phi (x) = x^2 + c$, where $c \in \mathbb{Z}$ and $c \neq 1$. Hamblen, Jones, and Madhu \cite{h15} later generalized the results to $\phi(x) = x^d+c$ (see also \cite{Benedetto}).  In other words, the primes $p$ for which $0$ is periodic (instead of pre-periodic) are of density zero.  These results imply that the primes $p$ for which the dynamical system consists of a single $p$-cycle modulo $p$ are of density zero. 

Let $S_{d,c}$ be the set of primes $p$ such that the dynamical system $\phi: \mathbb{Z}/p\mathbb{Z} \rightarrow \mathbb{Z}/p\mathbb{Z}$ consists of a single $p$-cycle. We show the following.

\begin{thm}
\label{thm:s}
Let $\phi(x) = x^d + c$, where $c,d \in \mathbb{Z}$ and $d \ge 2$.  Then whenever $d$ is even and $c$ is odd, $S_{d,c} = \{ 2 \}$; while if $d$ is even and $c$ is even, then $S_{d,c} = \emptyset$.
\end{thm}

Based on numerical data and heuristics, we conjecture that there are infinitely many such primes otherwise.

\begin{conj}
\label{conj:s}
$S_{d,c}$ is infinite whenever $d$ is odd.
\end{conj}

Using an analysis of the cycle structure of the permutation $x \mapsto x^d$ on $\mathbb{Z}/p\mathbb{Z}$, we are able to somewhat restrict the set $S_{d,c}$ as follows.

\begin{thm}\label{th:exceptional primes}
If $d \equiv 3 \pmod 4$, and $p \equiv 1 \pmod 4$ is prime, then $p \notin S_{d,c}$.
\end{thm}

For example, when $d$ is an odd power of $3$, we conclude that $S_{d,c}$ contains only primes congruent to $11 \pmod{12}$ (Corollary \ref{cor:11mod12}).

Theorem \ref{thm:s} is a consequence of our study of \emph{index divisibility} in dynamical sequences.  The question of index divisibility for a sequence $(a_n)_0^\infty$ seeks to characterize those integers $n \ge 1$ such that $n \mid a_n$.  It has a substantial history for Fibonacci and Lucas sequences \cite{AndreJeannin,HoggattBergum,Jarden,Smyth,Somer1,Somer2,Walsh}, and has also been studied for elliptic divisibility sequences \cite{Gotts,SilvermanStange} and general linear recurrences \cite{Pomerance}.  As another example, composite integers $n$ for which $n \mid a^n-a$ are called  pseudoprimes to the base $a$. 

Throughout, let $\phi(x) = x^d + c \in \mathbb{Z}[x]$ where $d \ge 2$, let $(W_n)$ denote the orbit of 0 under $\phi$, i.e. $W_n = \phi^n(0)$, and define
\begin{align*}
D_{d,c} := \{n \in \mathbb{Z} : n \geq 1, n  \mid  W_n \}, \atext{and} P_{d,c} := \{p \in D_{d,c}: p \text{ is prime}\}. 
\end{align*}
In the spirit of Smyth and of Silverman and Stange \cite{SilvermanStange, Smyth}, we represent $D_{d,c}$ by a directed graph that connects each element to its minimal multiples. To construct this \emph{index divisibility graph} $G$, initially let $1$ be in the vertex set $G_V$, then add vertices and edges to $G$ iteratively according to the the following rules.  

Let $v_p(x)$ denote the $p$-adic valuation of an integer $x$.  For each $n \in G_V$, adjoin the vertex $pn$ and the directed edge $(n,np)$ if
\begin{enumerate}
\item $p$ is a prime satisfying $v_p(\phi^n(0)) > v_p(n)$ (edge of \emph{type $1$}), or 
\item $p \in P_{d,c}$ satisfies $v_p(n) = 0$ (edge of \emph{type $2$}).
\end{enumerate}
Our main results provide a characterization of $D_{d,c}$ and $P_{d,c}$ in terms of this graph.

\begin{thm} \label{maintheorem}
Let $\phi(x) = x^d + c$, where $c,d \in \mathbb{Z}$ and $d \geq 2$. Let $G$ be the index divisibility graph corresponding to $\phi$, and let $G_V$ be the vertex set of $G$. Then $G_V = D_{d,c}$. 
\end{thm}

\begin{thm} \label{propertiesofP} 
Let $\phi(x) = x^d + c$, where $c,d \in \mathbb{Z}$ and $d \geq 2$. Then $P_{d,c}$ satisfies the following.
\begin{enumerate}
\item $2 \in P_{d,c}$.
\item Every divisor of $c$ is an element of $D_{d,c}$. In particular, if $p$ is prime and $p \mid c$, then $p \in P_{d,c}$. 
\item If $p$ is prime and $d \equiv 1 \pmod {p-1}$, then $p \in P_{d,c}$.
\end{enumerate}
\end{thm}

If $d$ is even, then we are able to fully characterize $P_{d,c}$.

\begin{thm} \label{eventhm} 
If $d$ is even, then
\begin{align*}
        P_{d,c} = \{2 \} \cup \{ p \text{ prime} : p \mid c \}.
\end{align*}
\end{thm}

Theorem \ref{thm:s} is an immediate consequence.

Two main tools we use in our investigation are the notions of a \emph{rigid divisibility sequence} and of a \emph{primitive prime divisor}.

An integer sequence $(a_n)$ is a \emph{rigid divisibility sequence} if for every prime $p$ the following two properties hold:
\begin{enumerate}
	\item if $v_p(a_n)>0$, then $v_p(a_{nk}) = v_p(a_n)$ for all $k \geq 1$, and
	\item if $v_p(a_n)>0$ and $v_p(a_m)>0$, then $v_p(a_{n}) = v_p(a_m) = v_p(a_{\text{gcd} (m,n)})$.
\end{enumerate}
In particular, rigid divisibility sequences are divisibility sequences.
%In particular, $(a_n)_{n=0}^\infty$ is a \emph{divisibility sequence} in the sense that $n \mid m$ implies $a_n \mid a_m$.

Rice \cite{r07} showed that for any polynomial $\phi \in \mathbb{Z}[x]$ of degree $d \geq 2$ where $0$ is a wandering point (i.e. of infinite orbit), the integer sequence $(\phi^n(0))$ is a rigid divisibility sequence if and only if the coefficient of the linear term of $\phi$ is zero. In particular, this means that the orbit of zero under $\phi(x) = x^d + c$, where $c,d \in \mathbb{Z}$ and $d \ge 2$, is a rigid divisibility sequence.

Given a sequence $(a_n)$ of integers, the term $a_n$ contains a \emph{primitive prime divisor} if there exists a prime $p$ such that $p \mid a_n$, but $p \nmid a_i$ for all $0<i<n$.  The study of primitive prime divisors dates back to Bang and Zsigmondy, who showed that every term of the sequence $(a^n-b^n)$, where $a,b \in \mathbb{Z}$ and $\gcd(a,b)=1$, has a primitive prime divisor \cite{Bang,Zsigmondy}.  Carmichael's Theorem asserts that the same is true for the Fibonacci numbers beyond the 12th term \cite{Carmichael}.  The \emph{Zsigmondy set} is the set of terms not having a primitive prime divisor; for the Fibonacci numbers, it is $\{1,2,6,12\}$.  Similarly, Silverman has shown that elliptic divisibility sequences have finite Zsigmondy sets \cite{SilvermanZsigmondy}.

Turning to dynamical sequences, Rice \cite{r07} showed that if $\phi(x) \in \mathbb{Z}[x]$ is a monic polynomial of degree $d \geq 2$, and $(\phi^n(0))$ is an unbounded rigid divisibility sequence, then all but finitely many terms contain a primitive prime divisor. Ingram and Silverman \cite{i09} generalized the results to rational functions over number fields. Doerksen and Haensch \cite{d12} extended upon this by finding explicit upper bounds on the Zsigmondy set for certain polynomial maps.

The following examples illustrate our results.
\begin{example}
Suppose $\phi (x) = x^2+3$. Then the orbit of 0 is
\begin{align*}
0, 3, 12, 147, 21612, 467078547, \ldots.
\end{align*}
Here,
\begin{align*}
D_{2,3} = \{1, 2, 3, 4, 6, 12, 21, 42, \ldots\} \atext{and} P_{2,3} = \{2,3\}
\end{align*}
by Theorems \ref{propertiesofP} and \ref{eventhm}. The index divisibility graph is shown in Figure \ref{fig:g23}.

\begin{figure}[ht]

\begin{tikzpicture}[>=stealth,every node/.style={circle}]

\node at (-8,0) (1) {$1$};
\node[draw,outer sep = 1mm] at (-6,1) (2) {$2$};
\node[draw,outer sep = 1mm] at (-6,-1) (3) {$3$};
\node at (-4.5,0) (6) {$6$};
\node at (-4,1.5) (4) {$4$};
\node at (-4,-1.5) (21) {$21$};
\node at (-2, 2) (72) {$7204$};
\node at (-2, 0.67) (12) {$12$};
\node at (-2, -0.67) (42) {$42$};
\node at (-2, -2) (14) {$147$};

\draw[->] (1) edge node[above left=-1mm,pos=.5,font = \tiny] {$1,2$} (2);
\draw[->] (1) edge node[above right=-1mm, pos=.5, font=\tiny] {$1,2$} (3);
\draw[->] (2) edge node[above right=-1mm, pos=.5, font=\tiny] {$1,2$} (6);
\draw[->] (2) edge node[above=-1mm, pos=.5, font=\tiny] {$1$} (4);
\draw[->] (3) edge node[above left=-1mm, pos=.5, font=\tiny] {$1,2$} (6);
\draw[->] (3) edge node[above=-1mm, pos=.5, font=\tiny] {$1$} (21);
\draw[->] (4) edge node[above=-1mm, pos=.5, font=\tiny] {$1$} (72);
\draw[->] (4) edge node[above right=-1mm, pos=.5, font=\tiny] {$1,2$} (12);
\draw[->] (6) edge node[above=-1mm, pos=.5, font=\tiny] {$1$} (12);
\draw[->] (6) edge node[above=-1mm, pos=.5, font=\tiny] {$1$} (42);
\draw[->] (21) edge node[above left=-1mm, pos=.5, font=\tiny] {$1,2$} (42);
\draw[->] (21) edge node[above=-1mm, pos=.5, font=\tiny] {$1$} (14);
\end{tikzpicture}

\caption{A portion of the index divisibility graph for $\phi(x) = x^2+3$. The circled vertices are elements of $P_{2,3}$, and edges are labeled by their type.} \label{fig:g23}
\end{figure}
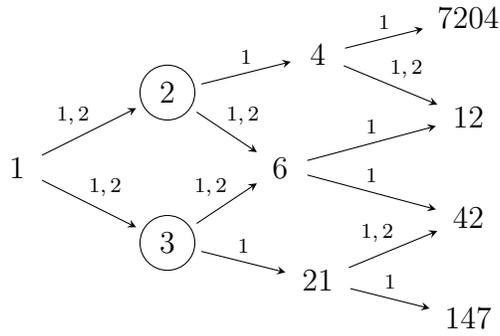

\end{example}

Notice in Figure \ref{fig:g23} that all type $2$ edges are also type $1$ edges. However, this is not always the case, as shown in Figure \ref{fig:g34}.

\begin{example} \label{ex:x^3+4}
Suppose $\phi (x) = x^3+4$. Then the orbit of 0 is:
\begin{align*}
0, 4, 68, 314436, \ldots.
\end{align*}
The index divisibility graph is illustrated in Figure \ref{fig:g34}.

\begin{figure}[ht]

\begin{tikzpicture}[>=stealth,every node/.style={circle}]

\node at (0,0) (1) {$1$};
\node[draw,outer sep = 1mm] at (2,2) (2) {$2$};
\node[draw,outer sep = 1mm] at (2,0) (3) {$3$};
\node[draw,outer sep = 1mm] at (2,-2) (11) {$\phantom{2}$};
\node at (2,-2) {$11$};
\node at (4,3) (34) {$\phantom{2}$};
\node at (4,3) {$34$};
\node at (4,1.5) (4) {$4$};
\node at (4,0) (6) {$6$};
\node at (4,-1.5) (22) {$\phantom{2}$};
\node at (4,-1.5) {$22$};
\node at (4,-3) (33) {$\phantom{2}$};
\node at (4,-3) {$33$};
\node at (6,2) (20) {$\phantom{2}$};
\node at (6,2) {$20$};
\node at (6,.5) (12) {$\phantom{2}$};
\node at (6,.5) {$12$};
\node at (6,-1) (44) {$\phantom{2}$};
\node at (6,-1) {$44$};

\draw[->] (1) edge node[above left=-1.5mm,pos=.5,font = \tiny] {$1,2$} (2);
\draw[->] (1) edge node[above=-3mm, pos=.5, font=\tiny] {$1,2$} (3);
\draw[->] (1) edge node[above=-1.5mm, pos=.5, font=\tiny] {$2$} (11);
\draw[->] (2) edge node[above=-1.5mm, pos=.5, font=\tiny] {$1$} (4);
\draw[->] (2) edge node[above=-1.5mm, pos=.5, font=\tiny] {$1$} (34);
\draw[->] (2) edge node[above=-1.5mm, pos=.5, font=\tiny] {$1,2$} (6);
\draw[->] (2) edge node[below=0mm, pos=.15, font=\tiny] {$2$} (22);
\draw[->] (3) edge node[above=-3mm, pos=.15, font=\tiny] {$1,2$} (6);
\draw[->] (3) edge node[above=-0.5mm, pos=.3, font=\tiny] {$2$} (33);
\draw[->] (11) edge node[above=-2.5mm, pos=.15, font=\tiny] {$1,2$} (22);
\draw[->] (11) edge node[above=-2.5mm, pos=.4, font=\tiny] {$1,2$} (33);
\draw[->] (4) edge node[above=-1.5mm, pos=.5, font=\tiny] {$1$} (20);
\draw[->] (4) edge node[above=-2.5mm, pos=.5, font=\tiny] {$1,2$} (12);
\draw[->] (4) edge node[above=-1.5mm, pos=.35, font=\tiny] {$2$} (44);
\draw[->] (6) edge node[above=-1.5mm, pos=.15, font=\tiny] {$1$} (12);
\draw[->] (22) edge node[above=-1.5mm, pos=.5, font=\tiny] {$1$} (44);

\end{tikzpicture}

\caption{A portion of the index divisibility graph for $\phi(x) = x^3+4$. The circled vertices are elements of $P_{3,4}$, and edges are labeled by their type.} \label{fig:g34}
\end{figure}
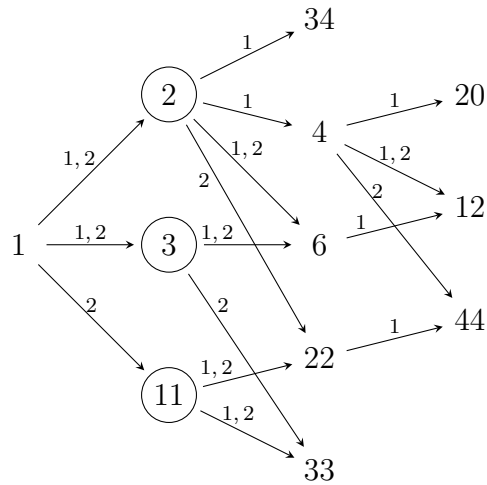
\end{example}

In Section 2, we study index divisibility and prove Theorems \ref{thm:s}, \ref{maintheorem}, \ref{propertiesofP}, and \ref{eventhm}. 

In Section 3, we classify the finiteness of $D_{d,c}$:  it is finite exactly when $(d,c) = (2,1), (2,-2)$, or $(d,1)$ where $d \ge 4$ is even (see Theorem \ref{th:D finite}).  

In Section 4, we study $P_{d,c}$ and its subset $S_{d,c}$ in the case where $d$ is odd, and prove Theorem \ref{th:exceptional primes}.

In Section 5, as a computational experiment, we find all pairs $(p,c)$, where $0 < c < p/2$ and $p \le 37619$, for which $p$ is in $S_{3,c}$ (see Figure \ref{fig:pvsc}). We provide some heuristics to support Conjecture \ref{conj:s}.

Finally, in Section 6, we ask the question, for a fixed $n$, of which pairs $(d,c)$ satisfy $n \in D_{d,c}$. 

\subsection*{Acknowledgements}
The authors would like to thank the Boulder Valley School District's Science Research Seminar program for making this research project possible, and would like to thank the first author's teacher in that program, Ryan O'Block, for his support.

\section{Index divisibility}

For the remainder of the paper, we maintain the notation presented in the introduction, namely $\phi(x) = x^d + c$ is an integral polynomial of degree at least 2, $W_n = \phi^n(0)$, $D_{d,c} = \{n \in \mathbb{Z} : n \geq 1, n  \mid  W_n \}$, and $P_{d,c}$ is the set of primes in $D_{d,c}$.

Before proceeding to the proofs, we identify two significant properties of $D_{d,c}$.

\begin{lem} \label{lem:smallestprimedivisor}
Suppose $n \in D_{d,c}$ and let $p$ be the smallest prime divisor of $n$. Then $p \in D_{d,c}$.
\end{lem}

\begin{proof}
Let $n \in D_{d,c}$, and write $n = pm$, where $p$ is the smallest prime factor of $n$. Then $p \mid W_n$ as $p \mid n$ and $n \mid W_n$. In particular, $0$ is periodic modulo $p$, so letting $b$ denote the period of $0$, it follows that $0 < b \le p$, $p \mid W_b$, and $b \mid n$. However, since $p$ is the smallest factor of $n$ greater than $1$, either $b = 1$ or $b = p$. If $b = p$, then $p \mid W_p$ as desired. Otherwise, if $b = 1$, then $p \mid W_1$, and hence $p \mid W_p$ since $W_1 \mid W_p$.
\end{proof}

\begin{lem} \label{lem:relprime}
If $a,b \in D_{d,c}$ are relatively prime, then $ab \in D_{d,c}$.
\end{lem}

\begin{proof}
Let $a$ and $b$ be relatively prime numbers in $D_{d,c}$. Since $(W_n)$ is a rigid divisibility sequence, we have that $a \mid ab$ implies $W_a \mid W_{ab}$, and $b \mid ab$ implies $W_b \mid W_{ab}$. Then because $a \mid W_a$, $a \mid W_{ab}$. Similarly, because $b \mid W_b$, we have $b \mid W_{ab}$. Since $a$ and $b$ are relatively prime, $ab \mid W_{ab}$, and so $ab \in D_{d,c}$. 
\end{proof}

We now proceed to prove Theorem \ref{maintheorem}.

\begin{proof}[Proof of \ref{maintheorem}.]
First we show $G_V \subseteq D_{d,c}$. To begin, we have $1 \mid W_1$, and so $1 \in D_{d,c}$.

Next we show that if $n \in D_{d,c}$ and $(n,np) \in G_E$ (the edge set of $G$), then $pn \in D_{d,c}$. We examine edges of type 1. Suppose there exist $n \in D_{d,c}$ and a prime $p$ such that $v_p(W_n) > v_p (n)$. Since $n \mid W_n$ and $v_p(W_n) > v_p (n)$, we see that $np \mid W_n$. Then since $(W_n)$ is a rigid divisibility sequence, $n  \mid np$ implies $W_n \mid W_{np}$. Thus $np \mid W_{np}$, and so $np \in D_{d,c}$. 

For edges of type 2, if $p \in P_{d,c}$ and $p \nmid n$, then $np \in D_{d,c}$ by Lemma \ref{lem:relprime}. Thus we have shown that $G_V \subseteq D_{d,c}$. 

We now proceed to show $D_{d,c} \subseteq G_V$. Suppose $n \in D_{c,d}$. To prove that $n \in G_V$, we show that $G$ contains a path from $1$ to $n$. If $n = 1$, there is nothing to show, so let $n = p_1^{\alpha_1} p_2^{\alpha_2} \cdots p_k^{\alpha_k}$ be the prime factorization of $n$, where $p_1<p_2<p_3< \cdots <p_k$. From Lemma \ref{lem:smallestprimedivisor}, we know that $p_1 \in D_{d,c}$, hence $(1,p_1)$ is an edge of type 2 in $G$.  

Now suppose $1 \le i \le k$ and $m$ is the largest divisor of $n$ supported on primes $p_j$, where $j < i$.  Then we have two observations:
\begin{enumerate}
        \item If $p_i \in P_{d,c}$, then $(m, mp_i)$ is an edge of type $2$.  If $p_i \notin P_{d,c}$, then $0$ has order less than $p_i$ modulo $p_i$, which implies $p_i \mid W_m$.  Then $v_{p_i}(W_m) > v_{p_i}(m)$, and so $(m, mp_i)$ is an edge of type $1$.
        \item Suppose $1 \le t < \alpha_i$.  Furthermore, $p_i \mid W_{mp_i} \mid W_{mp_i^t}$ (otherwise, the order of $0$ modulo $p_i$ exceeds $p_i$).  By rigid divisibility, $v_{p_i}(W_{mp_i^t}) = v_{p_i}(W_n) \ge \alpha_i > t = v_{p_i}(mp_i^t)$.  Therefore, we also have an edge of type 1: $(mp_i^t, mp_i^{t+1})$.  
\end{enumerate}
Combining these observations, starting with $1$, we have the following path of directed edges from $1$ to $n$:
\begin{align*}
1 &\xrightarrow[]{2} p_1 \xrightarrow[]{1} p_1^2 \xrightarrow[]{1} \cdots \xrightarrow[]{1} p_1^{\alpha_1}
 \\&\xrightarrow[]{1or2}  p_1^{\alpha_1} p_2 \xrightarrow[]{1} p_1^{\alpha_1} p_2^2 \xrightarrow[]{1} \cdots \xrightarrow[]{1} p_1^{\alpha_1} p_2^{\alpha_2}\\
 &\xrightarrow[]{1or2} \cdots \\ &\xrightarrow[]{1or2}  p_1^{\alpha_1} p_2^{\alpha_2}\cdots p_k \xrightarrow[]{1} p_1^{\alpha_1}  p_2^{\alpha_2}\cdots p_k^2 \xrightarrow[]{1} \cdots \xrightarrow[]{1} p_1^{\alpha_1}  p_2^{\alpha_2} \cdots p_k^{\alpha_k} = n. 
\end{align*}
Thus, $D_{d,c} \subseteq G_V$.
\end{proof}

We next prove Theorem \ref{propertiesofP}.

\begin{proof}[Proof of \ref{propertiesofP}.]
First, $W_2 = c^d+c = c^{d-1}(c+1)$. It follows that $2 \mid W_2$, and thus $2 \in P_{d,c}$. 

Next, to prove that every prime divisor of $c$ is an element of $D_{d,c}$, it suffices to show that $c \mid W_n$ for all $n$, for then it follows that every divisor of $c$ is an element of $D_{d,c}$. We proceed via induction on $n$. When $n=0$, we have $W_n = W_0 = 0$, hence $c \mid W_0$. Next assume the statement holds for $n=k$. Then $W_{k+1} = (W_k)^d + c$. Now since $c \mid W_k$, we have $c \mid (W_k)^d$, and so $c \mid W_{k+1}$.

For the third property, we show that if $p$ is prime, then $p \in P_{d,c}$ if $d \equiv 1 \pmod {p-1}$. Let $d = (p-1)k+1$, where $k \in \mathbb{Z}$. We have $\phi(x) = x^d + c = x^{(p-1)k+1} + c \equiv x+c \pmod p$, so $\phi^p(x) \equiv x+pc \equiv x \pmod p$. In particular, this means that $W_p = \phi^p(0) \equiv 0 \pmod p$, so $p \in P_{d,c}$.
\end{proof} 

We now prove Theorem \ref{eventhm}.

\begin{proof}[Proof of \ref{eventhm}.]
Let $d$ be even. We show that if $p$ is an odd prime, then $p \in P_{d,c}$ only if $p \mid c$.  

Suppose that $p \in P_{d,c}$. Then $p \mid W_p$, and the period of $0$ modulo $p$ is a divisor of $p$. If the period of $0$ is $1$, then $p \mid W_1 = c$. Otherwise if the period of $0$ is $p$, then $0$ has a unique preimage modulo $p$. In particular, $\sqrt[d]{-c} \equiv - \sqrt[d]{-c} \pmod p$.  Therefore $\sqrt[d]{-c} \equiv 0 \pmod p$, so $c \equiv 0 \pmod p$.\end{proof}

In conjunction with Theorem \ref{propertiesofP}, Theorem \ref{eventhm} provides a full characterization for $P_{d,c}$ when $d$ is even.   In particular, we can now prove Theorem \ref{thm:s}.

\begin{proof}[Proof of \ref{thm:s}]
When $c$ is odd, the orbit of $0$ has period $2$.  When $c$ is even, the orbit of $0$ has period $1$.  When $p \mid c$, the orbit of $0$ has period $1$.
\end{proof}

\section{Cardinality of $D_{d,c}$}

In Theorem \ref{th:D finite}, we identify all pairs $(d,c)$ for which $D_{d,c}$ is finite. But first, we note some simple infinite cases where $D_{d,c}$ is explicit.

\begin{lem} \label{lem:c=0,-1}
\
\begin{enumerate}
\item For all $d$, $D_{d,0}$ is the set of positive integers.
\item If $d$ is even, then $D_{d,-1}$ is the set of even positive integers.
\end{enumerate}
\end{lem}

\begin{proof}
If $c = 0$, then $W_n = 0$ for all $n$. When $c = -1$, then
\begin{align*}
W_n = \begin{cases*}
0 & if $n$ is even\\
-1 & if $n$ is odd.
\end{cases*}
\end{align*}
In both cases, the result is immediate.
\end{proof}

We now provide a simple yet sufficient condition for $D_{d,c}$ to be infinite.

\begin{lem} \label{lem:infinite}
If there exists $n \in D_{d,c}$ such that $n \ge 3$, then $D_{d,c}$ is infinite.
\end{lem}

\begin{proof}
Suppose that $n \in D_{d,c}$ for some $n \geq 3$. From \cite{d12}, $W_n$ contains a primitive prime divisor $p$, and since $p$ is primitive, we have $n \le p$. Since $p$ appears as a divisor in the sequence, $0$ is periodic modulo $p$, and the period cannot exceed $p$. Hence either $p = n$, or $p$ and $n$ are coprime.  If the latter holds, then, by Theorem \ref{maintheorem}, there is an edge of type $2$: $(n,np)$.  This implies that $n$ is not the largest element of $D_{d,c}$.  Therefore it suffices to consider the case $p = n$.

First, suppose that $d$ is even and $p = n$.  Then, by Theorem \ref{eventhm}, we have $p \mid c$, so that $p \mid W_n$ for all $n$.  This contradicts primitivity, so $d$ must be odd.

Therefore, suppose that $d$ is odd and $p=n$.  In this case, write $W_p = pm$ for some integer $m$. If $|m| > 1$, then for each prime divisor $q$ of $m$, the index divisibility graph contains the edge $(p, pq)$, hence $p$ is not the largest element of $D_{d,c}$.

Thus we are left considering the case $d$ is odd, $p=n$, and $W_p \in \{0, \pm p\}$.  However, we claim that this is not possible, by the growth of $W_n$.  For, since $d$ is odd, the signs of $W_n$, $W_n^d$, and $c$ are all the same by induction.  This implies that $|W_{n+1}| = |W_n^d + c| = |W_n^d| + |c| \ge |W_n|^d$.  In particular, since $|W_2| \ge 2$, we have $|W_n| > 2^{d^{n-2}}$.  (Here we use that $|c| \ge 1$. The case $c = 0$ is covered by Lemma \ref{lem:c=0,-1}.)  This rules out $|W_p| \le p$ for any $p \ge 3$.
\end{proof}

Consequently, $D_{d,c}$ is infinite in most cases.

\begin{thm} \label{th:D finite}
The set $D_{d,c}$ is finite if and only if either
\begin{enumerate}
\item $d$ is even and $c = 1$, or 
\item $d = 2$ and $c = -2$.
\end{enumerate} 
Moreover, if $D_{d,c}$ is finite, then $D_{d,c} = \{1,2\}$.
\end{thm}

\begin{proof}
By Theorem \ref{propertiesofP}, $c \in D_{d,c}$, hence it follows from Lemma \ref{lem:infinite} that $D_{d,c}$ is infinite whenever $|c| \ge 3$. Similarly, if $d$ is odd, then $3 \in P_{d,c}$ by Theorem \ref{propertiesofP}, and again $D_{d,c}$ is infinite. 

For the remainder of the proof, assume that $d$ is even. The cases $c = 0$ and $c = -1$ are handled by Lemma \ref{lem:c=0,-1}, leaving only the cases $c = 1$ and $c = -2$ to consider.

Suppose $c = 1$. In this case $W_1 = 1$ and $W_2 = 2$, and by Theorem \ref{eventhm}, we have $P_{d,1} = \{2\}$. Following the construction of the index divisibility graph, we have a single edge of type 2 emanating from the vertex 1 (the edge $(1,2)$), and there are no edges emanating from the vertex 2. Thus $D_{d,1} = \{1,2\}$.

Suppose $c = -2$. If $d = 2$, then $W_1 = -2$ and $W_2 = 2$. Similar to the previous case, the index divisibility graph only contains a single edge---the edge $(1,2)$---and hence $D_{2,-2} = \{1,2\}$.

Otherwise suppose $d \ge 4$. Then $W_2 = (-2)^d - 2 = -2((-2)^{d-1}+1)$, where $|(-2)^{d-1}+1| > 1$ and is odd. Hence $W_2$ has an odd prime divisor $p$, and therefore $(2,2p)$ is an edge of type 1 in the index divisibility graph. Since $2p \in D_{d,-2}$, it follows that $D_{d,-2}$ is infinite.
\end{proof}

\section{$S_{d,c}$ and $p$-cycles modulo $p$}

In Theorem \ref{propertiesofP}, we give a description of the set $P_{d,c}$. In the case that $d$ is even, Theorem \ref{eventhm} concludes that Theorem \ref{propertiesofP} completely determines $P_{d,c}$. However, when $d$ is odd, the conditions in Theorem \ref{propertiesofP} are insufficient to completely describe the set. This insufficiency is illustrated in Example \ref{ex:x^3+4} where we see that $11 \in P_{3,4}$, yet $11$ does not satisfy any of the conditions in Theorem \ref{propertiesofP}. 

Suppose then that $p \in P_{d,c}$ where both $p$ and $d$ are odd. As we have previously noted, if $p \in P_{d,c}$, then the period of $0$ in $\mathbb{Z}/p \mathbb{Z}$ is a divisor of $p$. If that period is $1$, then $p \mid c$, which Theorem \ref{propertiesofP} already accounts for. Therefore, the primes that are the exception are the odd primes for which $0$ has period $p$ modulo $p$. In other words, the primes of interest are the odd primes $p$ for which $x^d+c$ induces a single cycle of size $p$ in $\bbz/p\bbz$.

It is well known that $\pi(x) = x^d$ is a permutation of $\mathbb{Z}/p\mathbb{Z}$ if and only if $\gcd(d,p-1) = 1$. Hence under the same conditions, it follows that $x^d+c$ is a permutation of $\mathbb{Z}/p\mathbb{Z}$. In particular, we have $\phi = \tau^c \circ \pi$ (over $\mathbb{Z}/p\mathbb{Z}$), where $\tau(x) = x+1$. Since every $p$-cycle is an even permutation, we see that $\phi$ is a $p$-cycle only if $\pi$ is an even permutation. Equivalently, if $\pi$ is an odd permutation of $\mathbb{Z}/p\mathbb{Z}$, then $p \notin P_{d,c}$. 

We now use this observation to prove Theorem \ref{th:exceptional primes}.  For the remainder of this section, let $\ord_n m$ denote the order of $m$ in $(\mathbb{Z}/n\mathbb{Z})^*$. 

In order to understand the sign of $\pi$ as a permutation, we consider its cycle structure, which is given thusly.

\begin{lem}
Suppose $\pi(x) = x^d$ is a permutation of $\mathbb{Z}/p\mathbb{Z}$. Then $\pi$ has a cycle of length $m$ if and only if there exists a divisor $k$ of $p-1$ such that $\ord_k d = m$. Moreover, the number of cycles $N_m$ of length $m$ satisfies
\begin{align*}
mN_m = \sum_{i \mid m, i < m} iN_i.
\end{align*}
\end{lem}

\begin{proof}
See Lidl and Mullen \cite[Theorem 1]{l91}, as well as Ahmad \cite[Theorem 1]{a69} for a more general statement.
\end{proof}

In particular, letting $\varphi$ denote the Euler totient function, the theory of cyclic groups gives the following cycle structure.

\begin{lem} \label{lem:cycle}
Let $p$ be prime, and suppose $\gcd(d,p-1) \neq 1$. Then $x \mapsto x^d$ is a permutation of $\bbz/p\bbz$ with the following cycle structure:
\begin{enumerate}
\item $0$ is fixed, and
\item for each divisor $k$ of $p-1$, there are $\varphi(k)$ elements of $(\bbz/p\bbz)^*$ of order $\ord_k d$, i.e. the permutation contains $\varphi(k)/(\ord_k d)$ cycles of length $\ord_k d$ for each divisor $k$ of $p-1$.
\end{enumerate}
\end{lem}

The following Lemma will also prove useful.

\begin{lem} \label{lem:2 group}
Let $d$ be an odd integer, let $\mu = v_2(d-1)$, and let $\nu = v_2(d^2-1)-1$ (i.e. $\nu = \max\{v_2(d-1),v_2(d+1)\}$). Then
\begin{align*}
\ord_{2^k} d = \begin{cases}
1 & 0 \le k \le \mu\\
2 & \mu < k \le \nu\\
2^{k-\nu} & \nu < k.
\end{cases}
\end{align*}
\end{lem}

\begin{proof}
If $v_2(d-1) \ge k$, then $d \equiv 1 \pmod{2^k}$, hence $\ord_{2^k} d = 1$. If $v_2(d+1) \ge k > 1$, then $d \equiv -1 \pmod{2^k}$, hence $\ord_{2^k} = 2$. Otherwise $v_2(d^{2^j} - 1) = \nu+j$, and it follows that $2^{k-\nu}$ is the order of $d$.
\end{proof}

\begin{proof}[Proof of Theorem \ref{th:exceptional primes}]
Let $p$ be a prime where $p \equiv 1 \pmod 4$. We will show that $\pi(x) = x^d$ is an odd permutation of $\mathbb{Z}/p\mathbb{Z}$ if and only if $d \equiv 3 \pmod 4$, which by the discussion at the start of this section is sufficient to prove the theorem. Moreover, we assume that $\gcd(d,p-1) = 1$, as this is both necessary and sufficient for $\pi$ to be a permutation. 

The cycle type of $\pi$ is given in Lemma \ref{lem:cycle}. Let $N_k = \varphi(k)/(\ord_k d)$ be the number of cycles of length $k$ in $\pi$. Since a $k$-cycle is the product of $k-1$ transpositions, we see that $\pi$ may be written as a product of the following number of transpositions:
\begin{align*}
\sum_{k \mid p-1} N_k((\ord_k d) -1) &= \sum_{k\mid p-1} \varphi(k) - \sum_{k \mid p-1}N_k \\
&= p-1 -\sum_{k \mid p-1} N_k.
\end{align*}
It now suffices to determine when $\sum_{k\mid p-1} N_k$ is odd.

To count the cycles, write $p-1 = 2^\lambda\omega$, where $\omega$ is odd. Then
\begin{align*}
\sum_{k \mid p-1} N_k = \sum_{\delta \mid \omega}\sum_{0 \le i \le \lambda} N_{2^i\delta}.
\end{align*}

Consider first the sum over $\delta > 1$; we will show that this is even. Using the same notation as in Lemma \ref{lem:2 group}, let $\mu = v_2(d-1)$ and $\nu = v_2(d^2-1)-1$. Then for each $\delta$, we have
\begin{align*}
\sum_{0 \le i \le \lambda} N_{2^i\delta} &= N_\delta + N_{2\delta} 
+ \sum_{2 \le i \le \mu} \frac{\varphi(2^i \delta)}{\ord_{2^i\delta} d} 
+ \sum_{\mu < i \le \nu} \frac{\varphi(2^i \delta)}{\ord_{2^i\delta} d} 
+ \sum_{\nu < i \le \lambda} \frac{\varphi(2^i \delta)}{\ord_{2^i\delta} d}.
\end{align*}
Note that $N_\delta + N_{2\delta} =2N_\delta$ since
\begin{align*}
N_{2\delta} = \frac{\varphi(2\delta)}{\lcm(\ord_2 d,\ord_\delta d)} = \frac{\varphi(\delta)}{\ord_\delta d} = N_\delta.
\end{align*}

Next, $\ord_{2^i\delta} d = \lcm(\ord_{2^i} d,\ord_\delta d)$ by the Chinese remainder theorem. Moreover, $\ord_\delta d \mid \varphi(\delta)$ because $\varphi(\delta) = |(\bbz/\delta\bbz)^*|$, and $\ord_\delta d$ is the order of $d$ in $(\bbz/\delta\bbz)^*$. Hence
\begin{align*}
\sum_{2 \le i \le \mu} \frac{\varphi(2^i \delta)}{\ord_{2^i\delta} d} 
= \sum_{2 \le i \le \mu} \frac{2^{i-1}\varphi(\delta)}{\ord_{\delta} d} 
\equiv 0 \pmod 2,
\end{align*}
Now since $i \ge 2$, 
\begin{align*}
\sum_{\mu < i \le \nu} \frac{\varphi(2^i \delta)}{\ord_{2^i\delta} d} 
= \sum_{\mu < i \le \nu} \frac{2^{i-1}\varphi(\delta)}{\lcm(2,\ord_{\delta} d)} \equiv 0 \pmod 2,
\end{align*}
and similarly,
\begin{align*}
\sum_{\nu < i \le \lambda} \frac{\varphi(2^i \delta)}{\ord_{2^i\delta} d}
= \sum_{\nu < i \le \lambda} \frac{2^{i-1}\varphi(\delta)}{\lcm(2^{k-v},\ord_{\delta} d)} \equiv 0 \pmod 2.
\end{align*}
We conclude that the portion of the sum where $\delta > 1$ is even.

We are left to consider the contribution from $\delta = 1$. Here,
\begin{align*}
\sum_{0 \le i \le \lambda} N_{2^i} &= \sum_{0 \le i \le \lambda} \frac{\varphi(2^i)}{\ord_{2^i} d}\\
&= 2+\sum_{2\le i \le \lambda} \frac{2^{i-1}}{\ord_{2^i} d}\\
&\equiv \begin{cases*}
1 \pmod 2 & if $v_2(d-1)=1$\\
0 \pmod 2 & otherwise.
\end{cases*}
\end{align*}
Therefore, $\pi$ is odd if and only if $d \equiv 3 \pmod 4$, concluding the proof.
\end{proof}

\begin{cor}
\label{cor:11mod12}
If $p \in P_{3^k,c}$ and $k$ is odd, then either $p = 2$, $p \mid c$, or $p \equiv 11 \pmod{12}$.
\end{cor}

\begin{proof}
The cases $p=2$ and $p\mid c$ are due to Theorem \ref{propertiesofP}. Otherwise, if $p \in P_{3^k,c}$, $k$ is odd, and $p \nmid c$, then $x^{3^k} + c$ is a cyclic permutation of $\mathbb{Z}/p\mathbb{Z}$, and hence $p \not\equiv 1 \pmod 3$. Finally, $p \not\equiv 5 \pmod{12}$ by Theorem \ref{th:exceptional primes}.
\end{proof}

As evidenced in Example \ref{ex:x^3+4}, primes falling into this third category do exist.

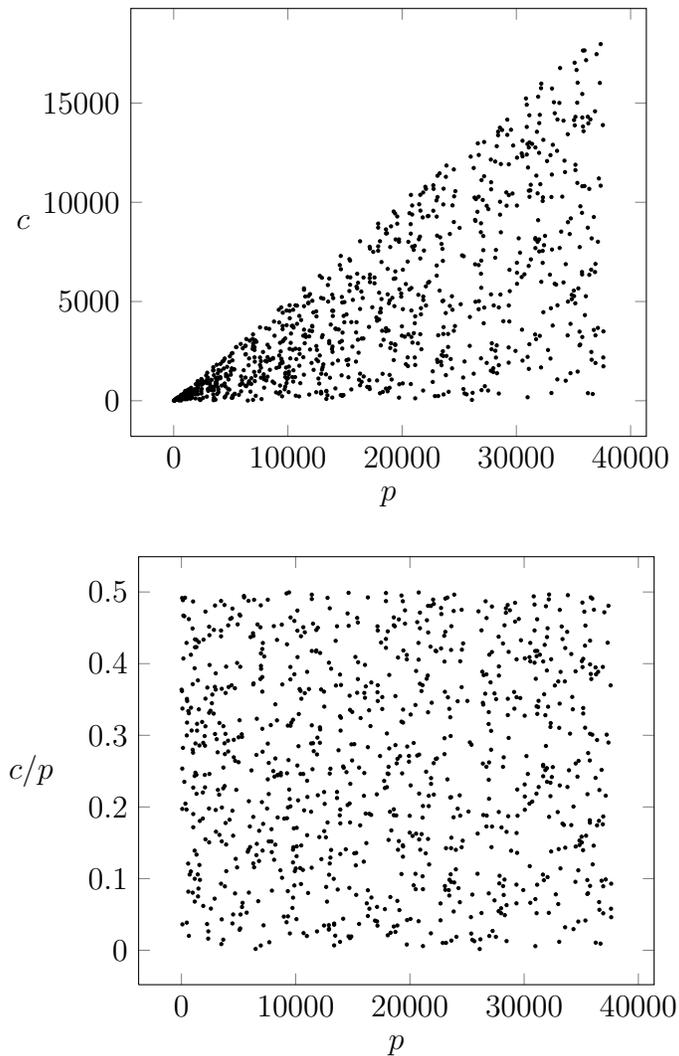
\begin{figure}[ht]
\begin{tikzpicture}
\begin{axis}[
scaled ticks=false,
title={\;},
xlabel={$p$},
tick label style={/pgf/number format/1000 sep={}},
ylabel={$c$}, 
ylabel style={rotate=-90},
ylabel style={xshift=-0.2cm}
]
\addplot+[scatter, only marks,
scatter/use mapped color={
    draw=black,
    fill=black,
},
mark size=0.6pt, mark=*, x=p,y=c]table{pvsc.dat};
\end{axis}
\end{tikzpicture}

\begin{tikzpicture}
\begin{axis}[
scaled ticks=false,
title={\;},
xlabel={$p$},
tick label style={/pgf/number format/1000 sep={}},
ylabel={$c/p$}, 
ylabel style={rotate=-90},
ylabel style={xshift=-0.2cm}
]
\addplot+[scatter, only marks,
scatter/use mapped color={
    draw=black,
    fill=black,
},
mark size=0.6pt, mark=*, x=p,y=cp]table{pvsc-scaled.dat};
\end{axis}
\end{tikzpicture}
\caption{The graph on the top shows a scatterplot of pairs $(p,c)$ such that $p \le 37619$ is prime and $p \in S_{3,c}$.  Below, the same scatterplot is scaled so that the pairs are of the form $(p,c/p)$.}
\label{fig:pvsc}
\end{figure}

\section{Heuristics and Experiment for $S_{d,c}$}

In this section, we provide some data to support Conjecture \ref{conj:s} in the case that $d=3$ and $d=9$, and consider some heuristic arguments.

In Figure \ref{fig:pvsc}, we plot all pairs $(p,c) \in [3,37619] \times [1,p/2]$ for which $p \in S_{3,c}$.  When $d$ is odd, if $p \in S_{d,c}$, then $p \in S_{d,c'}$ for any $c' \equiv \pm c \pmod p$ (Proposition \ref{prop:modsym}); hence the restriction to the interval $[1,p/2]$.  The data indicates that these pairs occur somewhat frequently and that the pairs $(p,c/p)$ seem to be distributed uniformly randomly in the rectangle $[1,37619]\times[0,0.5]$.

Based on this observation, as $c$ is restricted to integral values in $[1,(p-1)/2]$, let us adopt the following heuristic assumption:

\begin{hypothesis}
\label{hypo}
For any fixed $c$ the probability that a prime $p \ge 2|c|$ satisfies $p \in S_{3,c}$ is $2/(p-1)$.
\end{hypothesis}

Under this hypothesis, we may compute the expected number of pairs $(p,c)$ in the data set for any given $c$. Namely, the expectation for the number of data points for any fixed $c$ is
\begin{align*}
E_X(c) = \sum_{\substack{p \in [2|c|,X] \\ p \equiv 11 \bmod{12}}} \frac{2}{p-1}.
\end{align*}
In particular, $E_X(c) \rightarrow \infty$ as $X \rightarrow \infty$.
This quantity $E_X(c)$ is compared to the actual count for our data ($X = 37619$) in Figure \ref{fig:expectation}. 

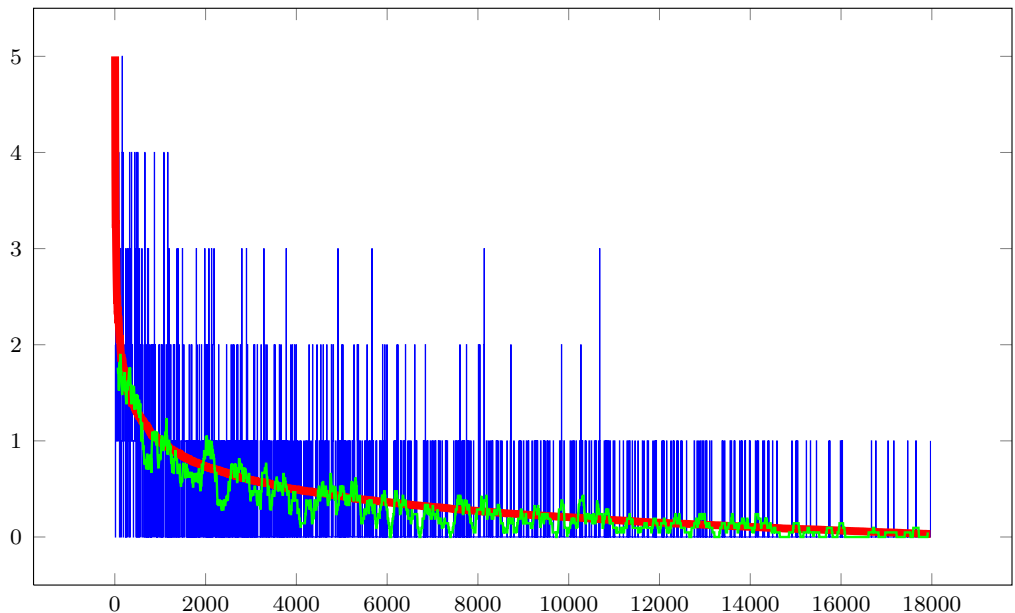
\begin{figure}[ht]
\begin{tikzpicture}[scale=1.05]
\begin{axis}[
width = \textwidth,
height=3.5in,
scaled ticks=false,
title={\;},
tick label style={font = \tiny, /pgf/number format/1000 sep={}},
]
\addplot[const plot,blue] table {rawc.dat};
%\addplot[Brown,line width=.3mm] table {7avg.dat};
\addplot[line width = 1mm,red] table {exp.dat};
\addplot[line width = .3mm,green] table {21avg.dat};
\end{axis}
\end{tikzpicture}
\caption{The data in Figure \ref{fig:pvsc} is collected in this figure by $c$ value in bins of size $6$. For each $k \in \bbn$, the value of the blue graph on the interval $[6(k-1),6k)$ is the number of pairs $(p,c)$ in the data for which $6(k-1) \le c < 6k$. At each point $x$, the green line is the average of the blue function over the interval $(x-60,x+60)$. The red line is the theoretical expectation under the assumption that the data is random, i.e. it is the graph of $E_{37619}(x)$.}
\label{fig:expectation}
\end{figure}

In Figure \ref{fig:3-exceptional}, we see that for approximately 60\% of $11 \pmod{12}$ primes, there exists a $c$ for which $p \in P_{3,c}$. Similarly in Figure \ref{fig:9-exceptional}, for just under 50\% of primes that are $5 \pmod 6$, there exists a $c$ such that $p \in P_{9,c}$.  (Corollary \ref{cor:11mod12} does not apply here; however we may restrict to $p \equiv 5 \pmod 6$ since if $p \equiv 1 \pmod 6$, then $\gcd(p-1,9) > 1$, so $x^9+c$ is not a permutation of $\mathbb{Z}/p\mathbb{Z}$.)  Of these primes, approximately two-thirds of them are $11 \pmod{12}$.

\begin{figure}[ht]
\begin{tikzpicture}
\begin{axis}[
scaled ticks=false,
title={\;},
xlabel={$X$},
tick label style={/pgf/number format/1000 sep={}},
ylabel={$\displaystyle\frac{\#U(X)}{\#P(X)}$}, 
ylabel style={xshift=-3mm},
ylabel style={rotate=-90}
]
\addplot+[black,only marks,mark size=0.6pt] table {ratio.dat};
\end{axis}
\end{tikzpicture}
\caption{Let $T(X) = \{p \le X : p \text{ prime}, p \equiv 11 \pmod{12}\}$ and $U(X) = \{p \in T(X) : p \in P_{3,c} \text{ for some $c$}\}$. The plot shows the ratio $\#U/\#T$ for $X \le 37619$.}
\label{fig:3-exceptional}
\end{figure}
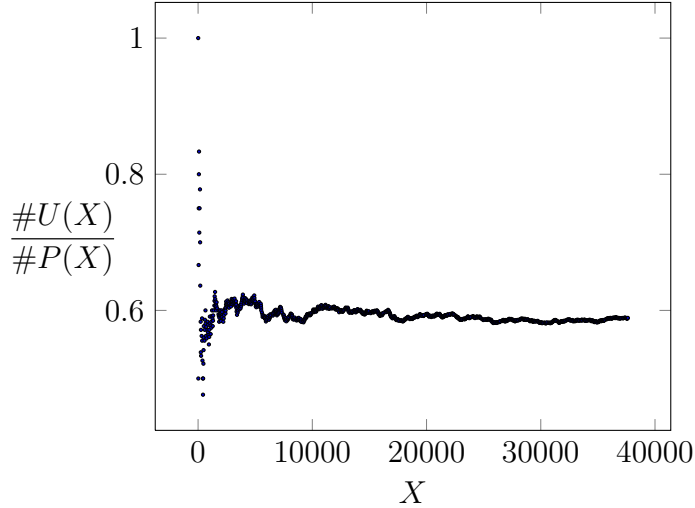

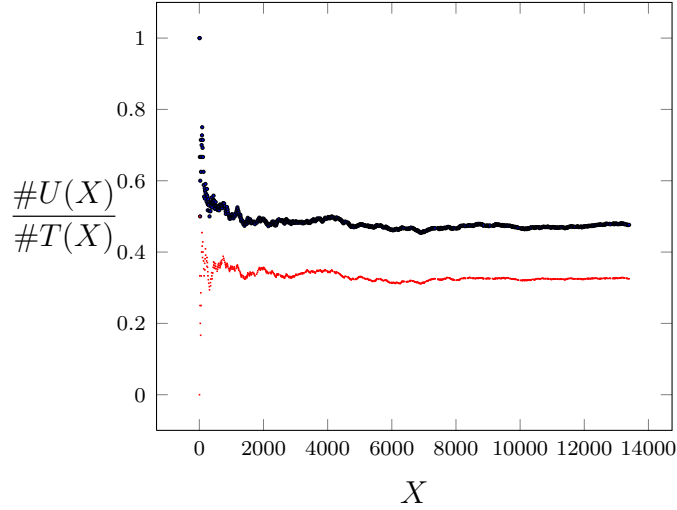
\begin{figure}[ht]
\begin{tikzpicture}
\begin{axis}[
scaled ticks=false,
title={\;},
xlabel={$X$},
tick label style={font=\tiny,/pgf/number format/1000 sep={}},
ylabel={$\displaystyle\frac{\#U(X)}{\#T(X)}$}, 
ylabel style={rotate=-90}
]
\addplot+[black,only marks,mark size=0.6pt] table {nine.dat};
\addplot+[only marks,mark size=0.1pt] table {nine2.dat};
\end{axis}
\end{tikzpicture}
\caption{Let $T(X) = \{p \le X : p \text{ prime}, p \equiv 5 \pmod 6\}$ and $U(X) = \{p \in T(X) : p \in P_{9,c} \text{ for some $c$}\}$. The plot shows the ratio $\#U/\#T$ for $X \le 13397$ [upper set]. Of the primes in $U(X)$, roughly two-thirds of them are $11 \pmod{12}$ [lower set].}
\label{fig:9-exceptional}
\end{figure}

We finish this section with a brief discussion of the relationship to dynatomic polynomials.  That is, fixing $p$, the number of $1 \le c \le p-1$ for which $p \in S_{3,c}$ is the number of non-zero roots of a dynatomic polynomial, as follows.  Given a family of maps $\phi_c$ (for us, $\phi_c(x) = x^3 + c$), write $\Psi_{n,0}(c) \in \mathbb{Z}[c]$ for the polynomial whose roots are those $c$ for which $0$ has period $n$, i.e.
\[
\Psi_{n,0}(c) = \phi_c^n(0).
\]
Then the $n$-th dynatomic polynomial $\Phi_{n,0}(c)$ is defined so that
\[
\Psi_{n,0} = \prod_{d \mid n} \Phi_{n,0}.
\]
In particular, $\Phi_{n,0}$ has as roots those $c$ such that $0$ has minimal period $n$ under the map $x^3+c$.  For more on these standard definitions, see \cite{SilvermanBook}.

With this setup, for a fixed $p$, the number of $1 \le c \le p-1$ for which $p \in S_{3,c}$ is equal to the number of non-zero roots of $\Phi_{p,0}(c)$ modulo $p$. 
One should note that the cyclotomic polynomials, which are considered analogous to the dynatomic polynomials, have very non-random behaviour in the corresponding situation, i.e. $\Phi_p(x)$ never has a root modulo $p$.

However, this raises an interesting general question.

\begin{question}
As the integer $n$ and prime $p$ vary, what is the splitting behaviour of $\Phi_{n,0}(c)$ modulo $p$?
\end{question}

For a polynomial $f \in \mathbb{Z}[x]$, the splitting behaviour modulo primes is understood via the Tchebotarev Density Theorem.  However, cyclotomic polynomials have different statistics than generic polynomials, since their Galois groups are necessarily abelian.  To what extent are dynatomic polynomials generic?

\section{Fixed $n$ and variable $c,d$}

In this section, we investigate $D_{d,c}$ from a different perspective:  for a fixed $n \in \mathbb{Z}$, in which $D_{d,c}$ does it appear? Let $H_n = \{(d,c) : n \in D_{d,c}\}$.

\begin{prop} \label{prop:modsym} 
For any integers $d,c \in \mathbb{Z}$, where $d \ge 2$, we have the following.
\begin{enumerate}
\item If $n \mid c$, then $(d,c) \in H_n$.
\item If $d \equiv 1 \pmod{n-1}$ and $n$ is prime, then $(d,c) \in H_n$.
\item If $(d,c_0) \in H_n$, then $(d,c) \in H_n$ whenever $c \equiv c_0 \pmod n$. Additionally, if $d$ is odd, then $(d,-c) \in H_n$ whenever $(d,c) \in H_n$.
\end{enumerate}
\end{prop}

\begin{proof}
The first two properties are immediate from Theorem \ref{propertiesofP}. For the third, set $\phi_c(x) = x^d + c$. If $c \equiv c_0 \pmod n$, then $\phi_c$ and $\phi_{c_0}$ are identical over $\mathbb{Z}/n\mathbb{Z}$. Hence $(d,c) \in H_n$ if and only if $(d,c_0) \in H_n$. Moreover, if $d$ is odd, then $\phi_{-c}(x) = -\phi_c(-x)$. Thus if $\phi^n_c(0) \equiv 0 \pmod n$, then $\phi_{-c}^n(0) \equiv 0 \pmod n$.
\end{proof}

Finally, we have a result regarding the powers of $d$ when $d$ is prime. 

\begin{thm}\label{prop:powersof3} 
If $d$ is prime, there exist $d$-adic integers $a_1, a_2, \ldots, a_{d-1}$, where $a_1 \equiv 1 \pmod d$, $a_2 \equiv 2 \pmod d, \ldots, a_{d-1} \equiv {d-1} \pmod d$, such that if $c \equiv 0, a_1, a_2, \ldots, a_{d-1} \pmod{d^n}$, then $(d,c) \in H_{d^n}$.
\end{thm}

\begin{proof}
Let $d$ be prime. From Theorem \ref{propertiesofP}, we have $d \in D_{d,c}$ for all $c \in \bbz$. In particular, $W_d \equiv 0 \pmod d$ for $c \equiv 0, 1, \ldots, d-1 \pmod d$. Considering $W_d$ as a function in $c$ (e.g. $W_d(c) = (\phi^{d-1}(0))^d + c$), we see that $\frac{d}{dc}W_d(c) \equiv 1 \pmod d$. Thus by Hensel's Lemma, each value modulo $d$ lifts to a unique $d$-adic solution. Namely, if $a_0, a_1, a_2, \ldots, a_{d-1} \in \mathbb{Z}_d$ are these lifts (where $a_i \equiv i \pmod d$) and $c \equiv a_i \pmod{d^n}$ for one of these $a_i$, then $W_d(c) \equiv 0 \pmod{d^n}$. It now follows from rigid divisibility that if $d^n \mid W_d$, then $d^n \mid W_{d^n}$. It is straightforward to verify that $a_0 = 0$.
\end{proof}

\bibliographystyle{plain}
\bibliography{IndexDivisibilityinDynamicalSequencesandCyclicOrbitsModp_Chen_Gassert_Stange}

\begin{figure}[ht]
\centering
\begin{tikzpicture}[>=stealth]
% main block
\foreach \twopow in {0,1,2} {
	\pgfmathtruncatemacro{\two}{2^\twopow}
	\pgfmathtruncatemacro{\three}{\two * 3}
	\pgfmathtruncatemacro{\eleven}{\two * 11}
	\pgfmathtruncatemacro{\thirtythree}{3 * \eleven}
	
	\draw ++(90:2*\twopow) node (\two) {$\two$}
		+(0:2) node (\eleven) {$\eleven$}
		++(30:1.2) node (\three) {$\three$}
		+(0:2) node (\thirtythree) {$\thirtythree$};
	
	\draw[->] (\two) edge (\three)
		(\two) edge (\eleven)
		(\three) edge (\thirtythree)
		(\eleven) edge (\thirtythree);
	}
	
\foreach \twopow in {0,1} {
	\pgfmathtruncatemacro{\two}{2^\twopow}
	\pgfmathtruncatemacro{\nexttwo}{2 * 2^\twopow}
	\pgfmathtruncatemacro{\three}{\two * 3}
	\pgfmathtruncatemacro{\nextthree}{2 * \two * 3}
	\pgfmathtruncatemacro{\eleven}{\two * 11}
	\pgfmathtruncatemacro{\nexteleven}{2 * \two * 11}
	\pgfmathtruncatemacro{\thirtythree}{3 * \eleven}
	\pgfmathtruncatemacro{\nextthirtythree}{2 * \thirtythree}
	
	\draw[->] (\two) edge (\nexttwo)
		(\three) edge (\nextthree)
		(\eleven) edge (\nexteleven)
		(\thirtythree) edge (\nextthirtythree);
	}
	
% 34 block	
\foreach \twopow in {1,2} {
	\pgfmathtruncatemacro{\two}{2^\twopow * 17}
	\pgfmathtruncatemacro{\three}{\two * 3}
	\pgfmathtruncatemacro{\eleven}{\two * 11}
	\pgfmathtruncatemacro{\thirtythree}{3 * \eleven}
	\pgfmathtruncatemacro{\twolabel}{2^\twopow}
	\pgfmathtruncatemacro{\threelabel}{3 * \twolabel}
	\pgfmathtruncatemacro{\elevenlabel}{11 * \twolabel}
	\pgfmathtruncatemacro{\thirtythreelabel}{33 * \twolabel}
	
	\draw (4.6,0)++(90:1.5*\twopow) node (\two) {$\twolabel p_1$}
		+(0:2) node (\eleven) {$\elevenlabel p_1$}
		++(30:1.5) node (\three) {$\threelabel p_1$}
		+(0:2) node (\thirtythree) {$\thirtythreelabel p_1$};
	
	\draw[->] (\two) edge (\three)
		(\two) edge (\eleven)
		(\three) edge (\thirtythree)
		(\eleven) edge (\thirtythree);
	}
	
\foreach \x in {34,102,374,1122} {
	\pgfmathtruncatemacro{\two}{2*\x}
	\draw[->] (\x) edge (\two);
	}	
	
	\draw[->] (2) edge[out=-30,in=180] (34);
	
% 20 block	
\foreach \twopow in {1} {
	\pgfmathtruncatemacro{\two}{4 * 5}
	\pgfmathtruncatemacro{\twolabel}{4}
	\pgfmathtruncatemacro{\three}{\two * 3}
	\pgfmathtruncatemacro{\threelabel}{\twolabel * 3}
	\pgfmathtruncatemacro{\eleven}{\two * 11}
	\pgfmathtruncatemacro{\elevenlabel}{11* \twolabel}
	\pgfmathtruncatemacro{\thirtythree}{3 * \eleven}
	\pgfmathtruncatemacro{\thirtythreelabel}{\twolabel * 33}
	
	\draw (4.6,3.5)++(90:1.4*\twopow) node (\two) {$\twolabel p_2$}
		+(0:2) node (\eleven) {$\elevenlabel p_2$}
		++(30:1.5) node (\three) {$\threelabel p_2$}
		+(0:2) node (\thirtythree) {$\thirtythreelabel p_2$};
	
	\draw[->] (\two) edge (\three)
		(\two) edge (\eleven)
		(\three) edge (\thirtythree)
		(\eleven) edge (\thirtythree);
	}	
	
	\draw[->] (4) edge[out=-30,in=-135] (20);

% 3p block	
\foreach \twopow in {0,1,2} {
	\pgfmathtruncatemacro{\two}{2^\twopow}
	\pgfmathtruncatemacro{\three}{3*\two}
	\pgfmathtruncatemacro{\threelabel}{3*\two*13}
	\pgfmathtruncatemacro{\eleven}{\three * 11}
	\pgfmathtruncatemacro{\elevenlabel}{\threelabel * 11}
	
	\draw (9.5,1)++(90:2*\twopow) node (\threelabel) {$\three p_3$}
		+(0:2) node (\elevenlabel) {$\eleven p_3$};
	
	\draw[->] (\threelabel) edge (\elevenlabel);
	}
	
\foreach \twopow in {0,1} {
	\pgfmathtruncatemacro{\two}{2^\twopow}
	\pgfmathtruncatemacro{\three}{3*\two}
	\pgfmathtruncatemacro{\threelabel}{3*\two*13}	
	\pgfmathtruncatemacro{\nextthreelabel}{2*\threelabel}
	\pgfmathtruncatemacro{\eleven}{\three * 11}
	\pgfmathtruncatemacro{\elevenlabel}{\threelabel * 11}
	\pgfmathtruncatemacro{\nextelevenlabel}{2*\elevenlabel}
	
	\draw[->] (\threelabel) edge (\nextthreelabel)
		(\elevenlabel) edge (\nextelevenlabel);
	}
	
	\draw[->] (3) edge[out=-15,in=180] (39);
\end{tikzpicture}
\caption{A graphical representation of a portion of $D_{3,4}$. Here $p_1 = 17$, $p_2 = 5$, and $p_3 = 26203$. To avoid clutter, not every edge between the vertices shown here is depicted.}  \label{fig:x^3+4}
\end{figure}
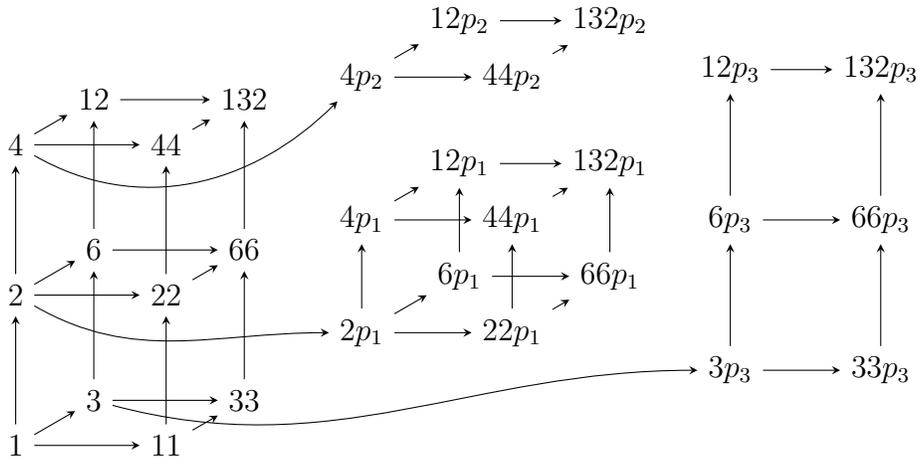

\end{document}